\documentclass[12pt,twoside,reqno]{amsart}

\newtheorem{thm}{Theorem}[section]
\newtheorem{lemma}{Lemma}[section]

\newtheorem{proposition}{Proposition}[section]

\def \R{{\Bbb R}}

\newcommand{\N}{{\mathbb N}}

\numberwithin{equation}{section}

\begin{document}

\title[Critical ZK equation]
{
Critical  2D Zakharov-Kuznetsov equation posed on  a half-strip
}

\author[
N.~A. Larkin]
{
	N.~A. Larkin \\
	Departamento de Matem\'atica,\\
	Universidade Estadual de Maringá\\
	87020-900, Maringá, Parana, Brazil}
\
\thanks
{email:nlarkine@uem.br}


\begin{abstract}
An initial-boundary value problem for the critical generalized 2D Zakharov-Kuznetsov equation posed on   the right half-strip  is considered.
Existence, uniqueness and exponential decay rate of global regular solutions for small initial data are established.
\end{abstract}

\maketitle

\section{Introduction}\label{introduction}

We are concerned with an initial-boundary value problem (IBVP) 
  for the critical Zakharov-Kuznetsov (ZK) equation  posed
  on the right half-strip
\begin{equation}
u_t+u^2u_x +u_{xxx}+u_{xyy}=0\label{zk}
\end{equation}
which is a two-dimensional analog of the   generalized Korteweg-de Vries (KdV) equation
\begin{equation}\label{kdv}
u_t+u^ku_x+u_{xxx}=0
\end{equation}
with  plasma physics applications \cite{zk} that has been intensively studied last years \cite{doronin1,familark,jeffrey,kaku}.                                  

Equations \eqref{zk} and \eqref{kdv} are typical examples
of so-called
dispersive equations attracting considerable attention
of both pure and applied mathematicians. The KdV
equation is  more studied in this context.
The theory of the initial-value problem
(IVP henceforth)
for \eqref{kdv} is considerably advanced today
\cite{tao,kato,ponce2,saut2}.

 Although dispersive equations were deduced for the whole real line, necessity to calculate numerically the Cauchy problem approximating the real line by finite intervals implies to study initial-boundary value problems posed on bounded and unbounded intervals \cite{doronin1,faminski2,familark,larluc1,larluc2,temam,temam2}.
What concerns (1.2) with $k>1, \;l=1$, called generalized KdV equations, the Cauchy problem  was studied in \cite{ martel,merle} and later in \cite{farah,Fonseka1,Fonseka2,ponce2}, where it has been established  that for $k=4$ (the critical case)  the problem is well-posed for small  initial data, whereas for arbitrary initial data solutions may blow-up in a finite time. The generalized Korteweg-de Vries equation was  studied for understanding the interaction between the dispersive term and the nonlinearity in the context of the theory of nonlinear dispersive evolution equations \cite{jeffrey,kaku,ponce2}. In \cite{lipaz}, the initial-boundary value problem for the generalized KdV equation with an internal damping posed on a bounded interval was studied in the critical case; exponential decay of weak solutions for small initial data has been established. In \cite{araruna}, decay of weak solutions in the case $l=2,\;k=2$ has been established.\\
 Recently, due to physics and numerics needs, publications on initial-boundary value
problems in both bounded and unbounded domains for dispersive equations  have been appeared
\cite{larluc1,larluc2,pastor,pastor2,pilod,saut2}. In
particular, it has been discovered that the KdV equation posed on a
bounded interval possesses an implicit internal dissipation. This allowed
to prove the exponential decay rate of small solutions for
\eqref{kdv} with $k=1$  posed on bounded intervals without adding any
artificial damping term \cite{doronin1}. Similar results were proved
for a wide class of dispersive equations of any odd order with one
space variable \cite{familark,larluc1,larluc2}.

The interest on dispersive equations became to extend their study for 
multi-dimensional models such as Kadomtsev-Petviashvili (KP)
and ZK equations. We call (1.1) a critical ZK equation by analogy with the critical KdV equation (1.2) for $k=4.$ It means that we could not prove the existence and uniqueness of global regular solutions without smallness restrictions for initial data similarly to the critical case for the KdV equation \cite{ farah, Fonseka1,Fonseka2,larluc2,martel,merle}.
As far as the ZK equation is concerned,
the results on both IVP and IBVP can be found in
\cite{faminski2,faminski3,farah,pastor,pastor2,marcia}. Our work has
been inspired by \cite{lipaz} where critical KdV equation with internal damping
posed on a bounded interval   was considered and exponential decay of weak solutions has been established.
We must note that solvability of initial-boundary value problems in classes of global regular solutions for the regular case of the 2D ZK equation\; ($uu_x$)\; has been established in \cite{doronin1,faminski3, lar1,lar2,lar5,larkintronco,marcia,temam,temam2} for arbitrary smooth initial data. On the other hand, for the 3D ZK equation, the convective term\; $uu_x$ , which is regular for the 2D ZK equation, corresponds to a critical case. It means that to prove the existence and uniqueness of global regular solutions one must put restrictions of small initial data \cite{lar3,lar4,larpad}.

The main goal of our work is to prove  for small initial data the existence and uniqueness
of global-in-time regular solutions for \eqref{zk} posed  on
bounded rectangles and the exponential decay rate of
these solutions.

The paper is outlined as follows: Section I is the Introduction. Section 2 contains  formulation
of the problem and auxiliaries. In Section \ref{existence}, Galerkin`s approximations
are  used to prove the existence and exponential decay of strong solutions. In Section 4,  regularity of strong solutions their  uniqueness and decay are established.

\section{Problem and preliminaries}\label{problem}

Let $(x,y)\equiv\;(x_1,x_2)\in\Omega$ and $\Omega$ \;be a domain in $\R^2$.
We use the usual notations of Sobolev spaces $W^{k,p}$, $L^p$ and $H^k$   and the following notations for the norms \cite{Adams}:

$$\| f \|_{L^p(\Omega)}^p = \int_{\Omega} | f  |^p\, d\Omega,\;\;
\| f \|_{W^{k,p}(\Omega)} = \sum_{0 \leq | \alpha| \leq k} \|D^\alpha f \|_{L^p(\Omega)},\;p\in(1,+\infty).$$
$$ \|f\|_{L^{\infty}(\Omega)}=ess\; sup_{\Omega}|f(x,y)|.$$

Let$$\mathcal{D}(\Omega) = \{ f \in C^\infty(\Omega);supp f \text{ is a compact set of} \;\Omega,\}$$
$$\R^+=\{t\in \R,\;\;t>0.\}$$
The closure of $\mathcal{D}(\Omega)$ in $ W^{k,p}(\Omega)$ is denoted by $ W_0^{k,p}(\Omega),\; H^k_0(\Omega)$ when $p=2.$\par
Let $L,B$ be finite positive numbers. Define
\begin{align*}
&D=\{(x,y)\in\mathbb{R}^2: \ x\in(0,L),\ y\in(0,B) \},\ \ \ Q=D\times \R^+;\\&
\gamma=\partial D \; \text{is a boundary of } \;D.
\end{align*}

Consider the following IBVP:
\begin{align}
A&u\equiv u_t+u^2u_x+u_{xxx}+u_{xyy}=0,\ \ \text{in}\;\; Q;
\label{2.1}
\\
&u_{\gamma\times t} =0,\; t>0;
\label{2.2}
\\
&u_x(L,y,t)=0,\ \ y\in(0,B),\ t>0;
\label{2.3}
\\
&u(x,y,0)=u_0(x,y),\ \ (x,y)\in D,
\label{2.4}
\end{align}
where $u_0:D\to\mathbb{R}$ is a given function.

Hereafter subscripts $u_x,\ u_{xy},$ etc. denote the partial derivatives,
as well as $\partial_x$ or $\partial_{xy}^2$ when it is convenient.
Operators $\nabla$ and $\Delta$ are the gradient and Laplacian acting over $D.$
By $(\cdot,\cdot)$ and $\|\cdot\|$ we denote the inner product and the norm in $L^2(D),$
and $\|\cdot\|_{H^k(D)}$ stands for the norm in $L^2$-based Sobolev spaces.

We will need the following result \cite{lady,lady2}.
\begin{lemma}\label{lemma1}
Let $u\in H^1(D)$ and $\gamma$ be the boundary of $D.$

If $u|_{\gamma}=0,$ then
\begin{equation}\label{2.5}
\|u\|_{L^q(D)}\le \beta\|\nabla u\|^{\theta}\|u\|^{1-\theta}.
\end{equation}
We will use frequently the following inequaliies:
$$\|u\|_{L^4(D)}\leq 2^{1/2}\|\nabla u\|^{1/2}\|u\|^{1/2},\;\;\|u\|_{L^8(D)}\leq 4^{3/4}\|\nabla u\|^{3/4}\|u\|^{1/4}.$$ 

If $u|_{\gamma}\ne0,$ then
\begin{equation}\label{2.6}
\|u\|_{L^q(D)}\le C_{D}\|u\|^{\theta}_{H^1(D)}\|u\|^{1-\theta},
\end{equation}
where $  \theta =2(\frac{1}{2}-\frac{1}{q}).$ 
\end{lemma}

\begin{lemma} \label{steklov} Let $v \in H^1_0(0,L).$ Then
	\begin{equation}\label{Estek} 
	\|v_x\|^2\geq \frac {\pi^2}{L^2}\|v\|^2.
	\end{equation}
\end{lemma}

\begin{proof} The proof is based on the Steklov inequality \cite{steklov}: let $v(t)\in H^1_0(0,\pi)$, then by the Fourier series $\int_0^{\pi}v_t^2(t)\,dt\geq\int_0^{\pi}v^2(t)\,dt.$
	Inequality \eqref{Estek} follows by a simple scaling.
\end{proof}

\begin{proposition}\label{prop1}
	Let for a.e. fixed $t$  $u(x,y,t)\in H^1_0(D)$ and $u_{xy}(x,y,t)\in L^2({D}).$ Then
	\begin{align}
		&\sup_{(x,y)\in{D}}u^2(x,y,t)\le 2\Big[ \|u\|^2_{H^1_0({D})}(t)+\|u_{xy}\|^2_{L^2({D})}(t)\Big]\notag\\&\leq 2\|u\|^2(t)_{H^2(D)\cap H^1_0(D)}.
	\end{align}
\end{proposition}
\begin{proof}
	For a fixed $x\in (0,L)$ and for any $y\in (0,B),$ it holds
	$$
	u^2(x,y,t)=\int_{0}^y\partial_su^2(x,s,t)\,ds\le \int_{0}^Bu^2(x,y,t)\,dy+\int_{0}^Bu_y^2(x,y,t)\,dy$$
	$$\equiv\rho^2(x,t).
	$$
	On the other hand,
	$$
	\sup_{(x,y)\in\mathcal{D}}u^2\le \sup_{x\in(0,L)}\rho^2(x)=\sup_{x\in (0,L)}\left|\int_0^x\partial_s\rho^2(s)\,ds\right|$$$$
	\le2\int_0^L\int_{0}^B\left(u^2+u_x^2+u_y^2+u_{xy}^2\right)\,dx\,dy \leq 2\|u\|^2_{H^2(D)\cap H^1_0(D)}.
	$$
	The proof of Proposition 2.1 is complete.
\end{proof}

\section{Existence theorem}\label{existence}

Define the space $W(D)$ with the norm $$\|u\|_{W(D)}=\|u\|_{H^2(D)\cap H^1_0(D)}+\|\Delta u_x\|.$$
\begin{thm}\label{theorem1}
Given  $u_0\in W(D)$ and $D$ such that $u_0|_{\gamma}=u_{0x}|_{x=L}=0$ and

\begin{equation}
\|u_0\|<min(\frac{1}{2},m),
\end{equation}where
\begin{align*}
&m<\Bigl(\frac{\pi^2}{4(1+L)^2}\bigl(\frac{1}{L^2}+\frac{1}{B^2}\bigr)\Big[5\times2^7(1+L)\|u_t(0)\|\bigl(1\\&+5^2\times2^{15}(1+L)^6\|u_t(0)\|\bigr)\Big]^{-1}\Bigr)^{1/3},\\&
\|u_t(0)\|\leq \|\Delta u_{0x}\|+\|u_0^2 u_{0x}\|\leq C\bigr(\|u_0\|_{H^2(D)\cap H^1_0(D)}\bigl).
\end{align*}
Then for all finite positive $B,\ L$ there exists a unique regular solution to
\eqref{2.1}-\eqref{2.4} such that
\begin{align*}
&u\in L^{\infty}(\R^+;H^2(D))\cap L^2(\R^+;H^3(D));\\
&\Delta u_x\in L^{\infty}(\R^+;L^2(D))\cap L^2(\R^+;H^1(D));\\
&u_t\in L^{\infty}(\R^+;L^2(D))\cap L^2(\R^+;H^1(D))
\end{align*}
and
\begin{align}\label{33.20}
&\|u\|_{H^2(D)}^2(t)+\|\Delta u_x\|^2(t)+\|u_t\|^2(t)\leq C(\|u_0\|_{W(D)})e^{(-\chi t)},\;\; t>0,\notag\\
&\text{where}\;\;\chi=\frac{\pi^2}{2(1+L)}\big[\frac{5}{L^2}+\frac{1}{B^2}\big];\\
&\int_{\R^+}\left\{\|u\|^2_{H^3(D)}(t)+\|\Delta u_x\|_{H^1(D)}^2(t)+\|u_x(0,y,t)\|^2_{H^2(0,B)}\right\}\,dt\notag\\
&\le C(\|u_0\|_{W(D)},\;L, B) .
\end{align}

\end{thm}

To prove this theorem, we will use the Faedo-Galerkin approximations. Let $w_j(y)$ be orthonormal in $L^2(D)$ eigenfunctions to the following Dirichlet Problem:
\begin{equation}
w_{jyy}+\lambda_jw_j=0, \;\;y\in (0,B); \;\; w_j(0)=w_j(L)=0;\;\;j\in {\N}.
\end{equation}
Define approximate solutions of (2.1)-(2.4) in the form:
\begin{equation}
u^N(x,y,t)=\sum_{j=1}^N g^N_j(x,t)w_j(y)
\end{equation}
and $g^N_j(x,t)$ are solutions to the following Korteweg-de Vries system:
\begin{align}
&g^N_{jt}+ g^N_{jxxx}-\lambda_jg^N_{jx}+\int_0^B |u^N|^2u^N_x w_j(y) dy=0,\\
&g^N_j(0,t)=g^N_j(L,t)=g^N_{jx}(L,t)=0;\; t>0,\\& g^N_j(x,0)=(u_{0N},w_j),\;\;x\in(0,L),
\end{align}
where $u_{0N}=\sum_{i=1}^N\alpha_{iN}w_i\to u_0\;\; \text{in} \;W(D)\cap H^1_0(D), \;\;j=1,...,N.$

Since each regularized KdV equation from (3.6) is not critical, it is known \cite{larluc1,larluc2, temam, temam2} that there exists a unique regular solution of (3.5)-(3.8) at least locally in $t$.

\noindent Our goal is to obtain global in $t$  a priori estimates for the $u^N$  independent of $t$ and $N,$ then to pass the limit as $N$ tends to $\infty$ getting a solution to (2.1)-(2.4).
\begin{lemma}\label{lem1}Under the conditions of Theorem 3.1,  the following independent of $N$ and $t$ estimates hold:
	\begin{equation}\label{est1}u^N \text{ is bounded in }L^{\infty}(\R^+; L^2(D))\cap L^2(\R^+; H^1(D))
\end{equation}
and
\begin{align}
	\|u^N\|^2(t)\leq &((1+x)^{1/2},u^N)^2(t)\leq ((1+x),u_0^2)e^{(-\chi t)} \notag\\\leq\frac{1+L}{4}e^{(-\chi t)},\;
	&\text{where}\;\;\chi=\frac{\pi^2}{2(1+L)}\big[\frac{5}{L^2}+\frac{1}{B^2}\big].
\end{align}
\end{lemma}

\begin{proof}
{\bf Estimate I.}
Multiply (3.6) by $g^N_j$, sum up over $j=1,...,N$ and integrate over $\Omega \times (0,t)$ to obtain

\begin{align} &\| u^N \|^2(t) + \int_0^t \int_0^B (u_{x}^N)^2(0,y,\tau)\, dy \,  \, d\tau  \notag\\&= \|u_{0}^N \|^2\leq \|u_0\|^2,\;\;t>0.
\end{align}

{\bf Estimate II.}\label{2-nd estimate}
Write the inner product
$$2\left(Au^N,(1+x)u^N\right)(t)=0,
$$
dropping the index $N$, in the form: 
\begin{align*}
\frac{d}{dt}\left((1+x),u^2\right)(t)
&+\int_{0}^Bu_x^2(0,y,t)\,dy
+3\|u_x\|^2(t)+\|u_y\|^2(t)\\
&=\frac12\int_{\mathcal{D}}u^4\,dx\,dy.
\end{align*}
Taking into account \eqref{2.5} and (3.11), we obtain
\begin{align*}
\frac12\int_{\mathcal{D}}u^4\,dx\,dy
&\le \frac12\|u\|^4_{L^4(\mathcal{D})}(t)
\le 2\|\nabla u\|^2(t)\|u_{0}\|^2(t).
\end{align*}
This implies
\begin{align}
\frac{d}{dt}&((1+x),u^2)(t)
+\frac12\|\nabla u\|^2(t)+\bigl(\frac12-2\|u_0\|^2\bigr)\|\nabla u\|^2(t)\notag\\&+2\|u_x\|^2(t)+\int_{0}^Bu_x^2(0,y,t)\,dy\le 0.\end{align}
Making use of (3.1) and Lemma 2.2, we get
$$\frac{d}{dt}\left((1+x),u^2\right)(t)+\frac{\pi^2}{2(1+L)}\big[\frac{5}{L^2}+\frac{1}{B^2}\big]((1+x),u^2)(t)\leq 0.$$
This gives

\begin{align}
\|u^N\|^2(t)\leq &((1+x)^{1/2},u^N)^2(t)\leq ((1+x),u_0^2)e^{(-\chi t)}\notag\\&\leq\frac{1+L}{4}e^{(-\chi t)}.
\end{align}
Returning to (3.12), we obtain

\begin{align}\label{3.6}
\left((1+x),|u^N|^2\right)(t)+&\int_0^t\int_{0}^B|u^N_{x}|^2(0,y,\tau)\,dy\,d\tau+\frac{1}{2}\int_0^t\|\nabla u^N\|^2(\tau)\,d\tau
\notag\\&\le ((1+x),u_0^2), \;\;t>0.
\end{align}
Moreover, we can  rewrite (3.12) as
\begin{align}
&4\|u^N_x\|^2(t)+\|\nabla u^N\|^2(t)+\int_{0}^B|u^N_x|^2(0,y,t)\,dy\leq 4|((1+x)u,u^N_t)(t)|\notag\\ &\leq4\|(1+x)^{1/2}u^N\|(t)\|(1+x)^{1/2}u^N_t\|(t).
\end{align}	
The proof of Lemma 3.1 is complete.
\end{proof}

\begin{lemma}
Under the conditions of Theorem 3.1, the following inequalities are true:
\begin{align}
&((1+x),|u^N_t|^2)(t)\leq C_0e^{(-\chi t)},\\
&\left((1+x),|u^N_ t|^2\right)(t)+\int_0^t\int_{0}^B(\partial^2_{x\tau}u^N)^2(0,y,\tau)\,dy\,d\tau
 	\notag\\
&+\frac12\int_0^t\|\nabla \partial_{\tau}u^N\|^2(\tau)\,d\tau\le C_0,\;\;t>0.
\end{align}
It follows from (2.1) that
\begin{align} &C_0 =(1+x,|u^N_t|^2)(0)\leq (1+L)\|u_t^N(0)\|^2\notag\\&
	\leq (1+L)\big[\|\Delta u_{0x}\|+\|u_0^2u_{0x}\|\Big]^2.
	\end{align}
\end{lemma}
\begin{proof}
Making use of (2.8) for $t=0$, we get
\begin{align}&\sup_D u^2_0(x,y)= \sup_D u^2(x,y,0)\leq 2\big(\|u_0\|^2_{H^1_0(D)}+\|u_{0xy}\|^2\Big)\equiv C_s.
	\end{align}
Substituting (3.19) into (3.18),we find
\begin{equation}(1+x,|u^N_t|^2)(0)\leq C_0(L,C_s,\|u_0\|_{W(D)}).
\end{equation}

{\bf Estimate III}\\
Dropping the index $N$, write the inner product
$$
2\left((1+x)u^N_t,\partial_t(Au^N\right)(t)=0
$$
as
\begin{align}\label{3.13}
	\frac{d}{dt}
	&\left((1+x),u_t^2\right)(t)+\int_{0}^Bu_{xt}^2(0,y,t)\,dy+3\|u_{xt}\|^2(t)	+\|u_{yt}\|^2(t)\notag\\
	&=2\left((1+x)u^2u_t,u_{xt}\right)(t)+2(u^2,u_t^2)(t).
\end{align}
Making use of Lemmas 2.1, 2.2, (3.15) and taking into account the first inequality of (3.1), we estimate
\begin{align*}
	I_1
	&=2\left((1+x)u^2u_t,u_{xt}\right)(t)\leq 2(1+L)\|u_{xt}\|(t)\|u\|^2(t)_{L^8(D)}\|u_t\|(t)_{L^4(D)}\\
	&\le 2^{9/2}(1+L)\|u_{xt}\|(t)\|u_0\|^{1/2}\|\nabla u\|^{3/2}(t)\|u_t\|^{1/2}(t)\|\nabla u_t\|^{1/2}(t)\\
		&\le 2^{9/2}(1+L)\|\nabla u_{t}\|^{3/2}(t)\|u_0\|^{1/2}\|\nabla u\|^{3/2}(t)\|u_t\|^{1/2}(t)\\&
		\leq \frac {3\delta}{4}\|\nabla u_t\|^2(t)+\frac{2^{16}(1+L)^4}{\delta^3}\|u\|^2(t)\|\nabla u\|^6(t)\|u_t\|^2(t)\\&
		\leq \frac {3\delta}{4}\|\nabla u_t\|^2(t)+\frac{2^{22}(1+L)^{11/2}}{\delta^3}\|u_0\|^5((1+x),u_t^2)^{3/2}(t)\|u_t\|^2(t).
		\end{align*}
	Here and henceforth\; $\delta$ is an arbitrary positive number.
Similarly,
\begin{align*}
	I_2
	&=2(u^2,u_t^2)(t)\le 2\|u\|^2(t)_{L^4(D)}\|u_t\|^2(t)_{L^4(D)}\\
	&\leq 2^3\|u\|(t)\|\nabla u\|(t)\|u_t\|(t)\|\nabla u_t\|(t)\\
	&\leq\frac{\delta}{2}\|\nabla u_t\|^2(t)+\frac{2^5}{\delta}\|u\|^2(t)\|\nabla u\|^2(t)\|u_t\|^2(t)\\&
	\leq\frac{\delta}{2}\|\nabla u_t\|^2(t)+\frac{2^7}{\delta}(1+L)^{1/2}\|u_0\|^3\|(1+x)^{1/2}u_t\|(t) \|(1+x),u^2_t)(t).
\end{align*}
Substituting $I_1, I_2$ into (3.21), making use of Lemma 2.2, taking into account the first inequality of (3.1) and setting $\delta=\frac15$ , we come to the inequality
\begin{align}\label{3.14}
&\frac{d}{dt}
((1+x),u_t^2)(t) +\int_{0}^Bu_{xt}^2(0,y,t)\,dy+2\|u_{xt}\|^2(t)	+\frac{1}{2}\|\nabla u_{t}\|^2(t)\notag\\& 
 + \big[\frac{1}{4(1+L)}\pi^2\big(\frac{1}{L^2}+\frac{1}{B^2}\big)-5\times2^7(1+L)^{1/2}\|u_0\|^3((1+x),u^2_t)^{1/2}(t)\{1\notag\\&+5^2\times 2^{15}(1+L)^{5}\|u_0\|^2((1+x),u^2_t)(t)\}\big]((1+x),u_t^2)(t)\leq 0.
\end{align}
Using (3.1), Lemma 2.2 and standard arguments,\;\cite{familark,larpad}, we obtain that
\begin{align*}
	&\frac{1}{4(1+L)}\pi^2\big(\frac{1}{L^2}+\frac{1}{B^2}\big)-5\times2^7(1+L)^{1/2}\|u_0\|^3((1+x),u^2_t)^{1/2}(t)\{1\notag\\&+5^2\times 2^{15}(1+L)^{5}\|u_0\|^2((1+x),u^2_t)(t)\}>0,\;\; t>0.
	\end{align*}
Returning to (3.21) and using the Steklov inequalities (2.7), we can rewrite it as

\begin{align}\label{3.14}
	\frac{d}{dt}
	&((1+x),u_t^2)(t) +\int_{0}^Bu_{xt}^2(0,y,t)\,dy\notag\\&+\frac{\pi^2}{2(1+L)}\big[\frac{5}{L^2}+\frac{1}{B^2}\big]((1+x),u^2_t)(t)\leq 0.
\end{align}
This implies
\begin{equation}
\left((1+x),u_t^2\right)(t)\le \left((1+x),u_t^2\right)(0)e^{-\chi t}\leq C_0e^{(-\chi t)}.
\end{equation}

Since (3.23) can be rewritten as
$$\frac{d}{dt}
((1+x),u_t^2)(t)+\int_{0}^B u_{xt}^2(0,y,t)\,dy+\frac12\|\nabla u_t\|^2(t)\leq 0,$$
integrating it we get

\begin{align}\label{3.15}
	&\left((1+x),u_ t^2\right)(t)+\int_0^t\int_{0}^B(\partial^2_{x\tau}u)^2(0,y,\tau)\,dy\,d\tau
	+\frac12\int_0^t\|\nabla \partial_{\tau}u\|^2(\tau)\,d\tau\notag\\
	&	\le (1+x),u_t^2)(0)=C_0.
\end{align}
Inequalities (3.24), (3.25) complete the proof of Lemma 3.2. 
\end{proof}
Returning to (3.15), we find
\begin{equation}
\|\nabla u^N\|^2(t)\leq Ce^{(-\chi t)},
\end{equation}
where the constant $C$ does not depend on $N, t>0.$

\begin{lemma} Under the conditions of Theorem 3.1, the following inequality holds:
\begin{align}
&\|\nabla u^N_y\|^2(t)+\int_0^B|u^N_{xy}|^2(0,y,t)dy\notag\\&
\leq C(L,\|u_0\|,\|u_t(0)\|e^{(-\chi t)}.
\end{align}
\end{lemma}
\begin{proof} Multiplying $j$-th equation of (3.6) by $\lambda_j$, and summing up the results over $J=1,...,N$, dropping the index $N$, we transfom the inner product
$$
-2\left((1+x)\partial^2_yu_,Au\right)(t)=0
$$
into the inequality
\begin{align}\label{3.7}
&3\|u_{xy}\|^2(t)+\|u_{yy}\|^2(t)+\int_{0}^{B}u_{xy}^2(0,y,t)\,dy
+\notag\\
&=\frac23\left((1+x)(u^3)_x,u_{yy}\right)(t)+2((1+x)u_{yy}, u_t)(t)\notag\\
&=-\frac23\left((1+x)(u^3)_{yx},u_{y}\right)(t)+2((1+x)u_{yy}, u_t)(t)\notag\\
&\leq \frac23|(u^3)_y,u_y+(1+x)u_{xy}|+2(1+L)\|u_t\|(t)\|u_{yy}\|(t)\notag\\
&\leq \delta\|\nabla u_{y}\|^2(t)+\frac{(1+L)^2}{\delta}\|u_t\|^2(t)\notag\\
&+2(u^2,u^2_y)(t)+2(1+x)u^2u_y,u_{xy})(t).
\end{align}
Making use of Lemma 2.1, we estimate
\begin{align*}
I_1&\equiv 2(u^2,u^2_y)(t)\leq2\|u\|^2_{L^4(D)}(t)\|u_y\|^2_{L^4(D)}(t)\notag\\
&\leq 4\|u\|(t)\|\nabla u\|(t)C^2_D\|u_y\|(t)\|\nabla u_y\|(t)\notag\\
&\leq \delta \|\nabla u_y\|^2(t)+\frac{4C^4_D}{\delta}\|u\|^2(t)\|\nabla  u\|^4(t),
\end{align*}
\begin{align*}
I_2&\equiv 2(1+x)u^2u_y,u_{xy})(t)\leq 2(1+L)\|u_{xy}\|(t)\|u^2\|_{L^4(D)}(t)\|u_y\|_{L^4(D)}\\
&\leq 2^{3/2}C_D^{1/2}(1+L)\|u_{xy}\|(t)\|u_y\|^{1/2}(t)\|\nabla u_y\|^{1/2}(t)\|u\|^2_{L^8(D)}\\
&\leq 2^{9/2}C_D^{1/2}(1+L)\|\nabla u_y\|^{3/2}(t)\|u_0\|^{1/2}\|\nabla u\|^{3/2}(t)\\
&\leq \frac{3\delta}{4}\|\nabla u_y\|^2(t)+\frac{2^{16}(1+L)^4 C_D^2}{\delta^3}\|u_0\|^2\|\nabla u\|^{6}(t).
\end{align*}
Substituting $I_1 ,I_2$ into (3.28),  we transform it into the following inequality:
\begin{align}
&\frac12\|\nabla u_y\|^2(t)+(\frac12-\frac{11\delta}{4})\|\nabla u_y\|^2(t)+\int_{0}^{B}u_{xy}^2(0,y,t)\,dy\notag\\
&\leq\frac{C}{\delta}\Big[ (1+L)^2\|u_t\|^2(t)+\|u_0\|^2\|\nabla  u\|^4(t)\notag\\&+\frac{2^{16}}{\delta^2}(1+L)^4\|u_0\|^2\|\nabla u\|^{6}(t)\Big].
\end{align}
Taking $11\delta=2$ and making use of (3.14), (3.15), we come to (3.27).
The proof of Lemma 3.3 is complete.
\end{proof}
\bigskip

\begin{lemma} Under the conditions of Theorem 3.1, the following inequality holds:
\begin{align}
&\int_0^t \Big[\|\nabla u^N_{yy}\|^2(\tau) +\int_0^B |u^N_{xyy}(0,y,\tau)|^2\;dy\Big]\;d\tau\notag\\&\leq C(D,\|u_0\|,\|u_t\|(0),\|u_{yy}\|(0)),\quad t>0.
\end{align}
\end{lemma}
\begin{proof}
{\bf Estimate IV}\label{4-th estimate}\\ Multiply each of the $j$-th equation of (3.6) by $\lambda_j^2$, sum up over $j=1,...,N$ and, dropping the index $N$, write the scalar product
$$
2\left((1+x)\partial^4_yu,Au\right)(t)=0
$$
in the form
\begin{align}\label{3.10}
&\frac{d}{dt}\left((1+x),u_{yy}^2\right)(t)
+3\|\partial^2_yu_x\|^2(t)+\|\partial^3_yu\|^2(t)+ \int_0^B u^2_{xyy}(0,y,t)dy\notag\\
&=-\frac23\left((1+x)u_{yy},(u^3)_{yyx}\right)(t).
\end{align}
Denote
\begin{align*}
I=&-\frac23\left((1+x)u_{yy},(u^3)_{yyx}\right)(t)
=\frac23\left(u_{yy},(u^3)_{yy}\right)(t)\\&+\frac23\left((1+x)u_{xyy},(u^3)_{yy}\right)(t)
\equiv I_1+I_2,
\end{align*}
where
$$
I_1=\frac23\left(u_{yy},(u^3)_{yy}\right)(t)=4(uu^2_y, u_{yy})(t)+2(u^2,u^2_{yy})(t)=I_{11}+I_{12}.
$$
By Proposition \ref{prop1} and (3.13), (3.18), (3.26), (3.27),
\begin{align}&\sup_{(x,y)\in{D}, t>0}u^2(x,y,t)\le 2\Big[ \|u\|^2_{H^1_0({D})}(t)+\|u_{xy}\|^2_{L^2({D})}(t)\Big]\notag\\&\leq Ce^{(-\chi t)}, \;\;t>0.\end{align}
Then

$$
I_{11}=2(u^2,u^2_{yy})(t)\le 2\sup_{(x,y)\in{D}}|u(x,y,t)|^2\,\|u_{yy}\|^2(t)
$$
$$
\le C\left((1+x),u^2_{yy}\right)(t)\leq Ce^{(-\chi t)}
$$
and 
$$
I_{12}=2(u_{yy},u u^2_y)(t)\le2\sup_{D}|u(x,y,t)|\|u_{yy}\|(t)\,\|u_y\|^2_{L^4({D})}(t)
$$
$$
\le 2C^2_{{D}}\sup_{D}|u(x,y,t)|\|u_{yy}\|(t)\|u_y\|(t)\|u_y\|_{H^1({D})}(t)$$$$\leq Ce^{(-\chi t)}.
$$

Similarly,
\begin{align*}
I_2&=\frac23((1+x)u_{xyy},(u^3)_{yy})(t)=2((1+x)u_{xyy},(u^2u_y)_y)(t)\\
&\leq 4|((1+x)u_{xyy},uu_y^2)(t)|+2|((1+x)u_{xyy},u^2u_{yy})(t)|\\&\leq 
4|((1+x)u_{xyy},uu_y^2)(t)|+2\sup_Du^2(x,y,t)|((1+x)u_{xyy},u_{yy}(t)|\\
&\leq 4(1+L)\sup_D|u(x,y,t)||u_{xyy}|\|u_y\|^2(t)_{L^4(D)}\\&+2(1+L)\sup_Du^2(x,y,t)\|u_{yy}\|(t)\|u_{xyy}\|(t)|\\&
\leq \delta \|u_{xyy}\|^2(t)+\frac{C(D)}{\delta}\Big[\sup u^2(x,y,t)\|u_y\|^2(t)\|\nabla u_{y}\|^2(t)\\
&+\sup_D u^4(x,y,t)\|u_{yy}\|^2(t)\Big].
\end{align*}
Making use of Proposition 2.1, (3.31) and Lemmas 3.2, 3.3, we get
$$I_2\leq \delta \|u_{xyy}\|^2(t)+\frac{C(L,\|u_0\|,\|u_t(0)\|)}{\delta}e^{(-\chi t)}.$$
Taking $\delta=\frac12$ and substituting $I_1, I_2 $ into (3.25), we find that

\begin{align*}
	&\frac{d}{dt}\left((1+x),u_{yy}^2\right)(t)
	+\frac12\|\nabla u_{yy}\|^2(t)+ \int_0^B u^2_{xyy}(0,y,t)dy\notag\\
	&\leq C(D,\|u_0\|,\|u_t(0)\|)e^{(-\chi t)}.
\end{align*}
Simple integration completes the proof of Lemma 3.4.
\end{proof}

Taking into account (3.3),(3.4), write (3.6) in the form
$$\int_0^B\bigl(u^N_{xxx}-\lambda_j u^N_x\bigr)w_j dy=-\int_0^B\bigl[u^N_t +|u^N|^2u^N_x\bigr]w_jdy.$$

Multiplying it by $g_{jxxx}-\lambda_j g_{jx}$, summing over $j=1,...,N$ and integrating with respect to $x$ over $(0,L),$ we obtain
$$\|\Delta u^N_x\|^2(t)\leq \|u^N_t\|(t)\|\Delta u^N_x\|(t)+\||u^N|^2u^N_x\|(t)\|\Delta u^N_x\|(t)$$
or
$$\|\Delta u^N_x\|(t)\leq \|u^N_t\|(t))+\||u^N|^2u^N_x\|(t).$$
Making use of (2.8), (3.24),  we get
\begin{equation}
\|\Delta u^N_x\|(t)\leq Ce^{(-\frac{\chi t}{2})}
\end{equation}
where the constant $C$ does not depend on $N,t>0.$

\subsection{Passage to the limit as $N\to \infty$}\label{limit}

Since the constants in Lemmas 3.1-3.4  and (3.33) do not depend on
$N,t>0,$ then making use of the standard arguments, see \cite{lady,temam2}, one may pass to the limit as $N\to \infty$ in (3.6) to obtain for all $\psi(x,y)\in L^2(D):$
\begin{equation}\label{3.22}
\int_{\mathcal{D}}\left[u_t+u^2u_x+\Delta u_x\right]\psi\,dx\,dy=0.
\end{equation}
Taking into account Lemmas 3.1-3.4, we establish the following result:

\begin{lemma}\label{lema1}
Let all the conditions of Theorem \ref{theorem1} hold. Then there exists a strong solution $u(x,y,t)$ to \eqref{2.1}-\eqref{2.4} such that
\begin{align}\label{3.23}
&\| u\|_{H^1_0(D)}^2(t)+\|\nabla u_y\|^2(t)+\|u_t\|^2(t)+\|\Delta u_x\|^2(t)\notag\\&+\|u_{x}(0,y,t)\|^2_{H^1_0(0,B)}\leq Ce^{(-\chi t)},\\
&\!\int_0^t\!\Big\{\|\nabla u_{yy}\|^2(\tau)+\|\Delta u_x\|^2_{L^2(D)}(\tau)+\|u_t\|^2_{H^1_0(D)}(\tau)\notag\\&+\|u_{x}(0,y,\tau)\|^2_{H^2(0,B)}d\tau\Big\}\notag\\
&\le C(D,\|u_0\|,\|u_t\|(0),\|u_{yy}\|(0)),\; t>0.
\end{align}
\end{lemma}

\section{More regularity}

In order to complete the proof of the existence part of Theorem \ref{theorem1}, it suffices to show that
$$
u\in L^{\infty}\left(\R^+; H^2({D})\right)\cap u\in L^2\left(\R^+; H^3({D})\right), $$
$$u_t\in L^{\infty}\left(\R^+; L^2({D})\right)\cap u\in L^2\left(\R^+; H^1({D})\right).
$$
These inclusions will be proved in the following lemmas.

\begin{lemma}\label{lema2}
A strong solution from Lemma \ref{lema1} satisfies the following inequality:
\begin{align}\label{3.24}
&\int_{\R^+}\left\{\|u\|^2_{H^3({D})}(t)+\|\Delta u_x\|^2_{H^1({D})}(t)\right\}dt\notag\\&\le C\bigl(D,\|u_0\|,\|u_t\|(0),\|u_{yy}\|(0)\bigr)e^{(-\chi t)}.
\end{align}

\end{lemma}
\begin{proof}
Taking into account \eqref{3.23}, (3.36) and Proposition \ref{prop1}, we write \eqref{3.22} in the form
$$\Delta u_x=-u_t-u^2u_x\equiv f(x,y,t)\in L^{\infty}(\R^+;L^2(D)),$$
$$u_x(0,y,t)\equiv \phi(y,t)\in L^2\left(\R^+;H^2(0,B)\right)\cap L^{\infty}\left(\R^+;H^1_0(0,B)\right),$$
$$u_x(x,0,t)=u_x(x,B,t)=u_x(L,y,t)=0.$$
Denote $\Phi(x,y,t)=\phi(y,t)(1-x/L)$. Obviously, $$\Phi\in L^2\left(\R^+;H^2({D})\right).$$ Then the function
$$
v=u_x-\Phi(x,y,t)
$$
solves in ${D}$ the elliptic problem
\begin{equation*}
\Delta v=f(x,y,t)-\Phi_{yy}(x,y,t)\in L^2(R^+;L^2(0,D)),\ \ v|_{\gamma}=0
\end{equation*}
which admits a unique solution $v\in L^2\left(\R^+;H^2({D})\right)$, see \cite{lady}.
Consequently, $u_x\in L^2\left(\R^+;H^2({D})\right).$ Therefore (3.36) implies (4.1). This completes the proof of Lemma 4.1..
\end{proof}

\begin{lemma}\label{lema3}
A strong solution given by Lemma \ref{lema1} satisfies the following inequality:
\begin{align}\label{3.25}
&\|u\|^2_{H^2({D})}(t)+\|\Delta u_x\|^2(t)\notag\\&\le C(D,\|u_0\|,\|u_t\|(0),\|u_{yy}\|(0))e^{(-\chi t)},\;\;t>0.
\end{align}

\end{lemma}
\begin{proof}
Making use of (3.35) and acting as by the proof of Lemma 4.1, we get

\begin{align*}
	&\Delta u_x=-u_t-u^2u_x\equiv f(x,y,t)\in L^{\infty}(\R^+;L^2(D)),\\
	&u_x(0,y,t)\equiv \varphi(y,t)\in  L^{\infty}\left(\R^+;H^1(0,B)\right),\\
	&u_x(x,0,t)=u_x(x,B,t)=u_x(L,y,t)=0.
\end{align*}
Denote $\Phi(x,y,t)=\varphi(y,t)(1-x/L)$. Obviously, $$\Phi\in L^{\infty}\left(\R^+;H^{1}_0(\mathcal{D})\right).$$ Then the function
$$
v=u_x-\Phi(x,y,t)
$$
solves in ${D}$ the elliptic problem
\begin{equation}
	\Delta v=f(x,y,t)-\Phi_{yy}(x,y,t)\in L^{\infty}\left(\R^+;H^{-1}(D)\right),\ \ v|_{\gamma}=0.
	\end{equation}

By the elliptic equations  theory \cite{lady},\; there exists a unique weak solution to (4.3),
$$v\in L^{\infty}\left(R^+;H^{1}_0(D)\right).$$
Consider the scalar product$$
-(v,\Delta v)(t)=-(v,f-\Phi_{yy})(t)=-(v,f)(t)-(v_y,\Phi_y)(t)$$
that can be rewritten in the form
$$\|v_x\|^2(t)+\|v_y\|^2(t)\leq	\frac{1}{2}\big[\|f\|^2(t)+\|v\|^2(t)+\|\Phi_y\|^2(t)+\|v_y\|^2(t)\Big].$$
This implies
\begin{equation}
\|v\|^2_{H^1(D)}(t)\leq C[\|f\|^2(t)+\|\Phi_y\|^2(t)]\leq Ce^{(-\chi t)}.
\end{equation}
Hence, $u_x=v+\Phi \in L^{\infty}(\R^+;H^1(D))$ and since $ u,\;u_y\in L^{\infty}\left(R^+;H^{1}(D)\right),$\\
then$$u\in L^{\infty}\left(\R^+;H^2(D)\cap H^1_0(D)\right).$$
This completes the proof of Lemma 4.2.
\end{proof}
\
\begin{lemma} The strong solution from Lemmas 3.5-3.7 is uniquelly defined.
	\end{lemma}

\begin{proof}
Let $u_1$ and $u_2$ be two distinct solutions to
\eqref{2.1}-\eqref{2.4}. Then $z=u_1-u_2$ solves the
following IBVP:
\begin{align}
Az&\equiv z_t++z_x+\frac12(u_1^3-u_2^3)_x+\Delta z_x=0\ \ \text{in}\ {Q}\;\;t>0,\label{5.1}\\
&z(0,y,t)=z(L,y,t)=z_x(L,y,t)=z(x,0,t)\notag\\
&=z(x,B,t)=0,\\
&z(x,y,0)=0,\ \ (x,y)\in{D}.\label{5.3}
\end{align}
From the scalar product $$2\left(Az,(1+x)z\right)(t)=0,$$ we infer
\begin{align}
\frac{d}{dt}&((1+x),z^2)(t)(t)+3\|z_x\|^2(t)+\|z_y\|^2(t)+\int_{0}^Bz_x^2(0,y,t)\,dy\notag\\
&=-2((1+x)(u^3_1-u^3_2)_x,z)(t)=2((u_1^2+u_1u_2+u_2^2)z,z\notag\\&+(1+x)z_x)(t)\leq 2\|z_x\|^2(t)+C(M)((1+x),z^2)(t),
\end{align}

where $M=\sup_D|u_1^2+u_1u_2+u_2^2|(x,y,t)$. Due to Proposition 2.1 and (4.2), $M$ does not depend on $t>0.$ Hence (4.8) becomes
$$\frac{d}{dt}((1+x),z^2)(t)(t)\leq C(M)((1+x),z^2)(t).$$
Since $z(x,y,0)\equiv 0,$ by the Gronwall lemma,$$\|z\|^2(t)\leq ((1+x),z^2)(t)\equiv 0,\;\;t>0.$$

The proof of Lemma 4.3 is  complete.
\end{proof}

{\bf Conclusions.} An initial-boundary value problem for the 2D critical generalized Zakharov-Kuznetsov equation posed on  rectangles has been considered. Assuming small initial data, the existence of a regular global solution, uniqueness and exponential decay of  $\|u\|(t)_{H^2(D)}$ have been established.

{\bf Acknowledgements.} 
This research has been supported  by Funda\c{c}\~ao Araucaria , Parana, Brazil; convenio No 307/2015, Protocolo No 45.703.\\ The author appreciates profound and concrete comments of the reviewer.

{\bf Data Availability Statement.}
This research does not contain any data necessary to confirm.

\end{document}